\documentclass[11pt,a4paper]{article}

\usepackage[authoryear,longnamesfirst]{natbib}
\usepackage{amsmath,amssymb,amsthm,mathtools}
\usepackage{booktabs,array,multirow}
\usepackage{graphicx}
\usepackage{microtype}
\usepackage{float}
\usepackage{xcolor}
\usepackage{url}
\usepackage[margin=1in]{geometry}
\usepackage{enumitem}
\usepackage{hyperref}
\usepackage{authblk}

\newtheorem{theorem}{Theorem}[section]
\newtheorem{proposition}[theorem]{Proposition}
\newtheorem{lemma}[theorem]{Lemma}
\newtheorem{corollary}[theorem]{Corollary}
\theoremstyle{definition}

\theoremstyle{remark}
\newtheorem{remark}[theorem]{Remark}

\newcommand{\T}{\mathcal T}
\newcommand{\Pp}{\mathbb P}

\newcommand{\R}{\mathbb R}
\newcommand{\rank}{\operatorname{rank}}
\newcommand{\kernel}{\operatorname{ker}}
\newcommand{\spann}{\operatorname{span}}
\newcommand{\bbB}{\mathbf B}
\newcommand{\bbC}{\mathbf C}

\newcommand{\eps}{\varepsilon}

\def\tsc#1{\csdef{#1}{\textsc{\lowercase{#1}}\xspace}}
\tsc{WGM}
\tsc{QE}

\title {Geometry-dependent rank defect in  \texorpdfstring{$C^1$}{C1} cubic spline space}  

\author{%
Xinyu Wu \quad Jiansong Deng\\
\small \texttt{xinyu97@ustc.edu.cn}\quad
\texttt{dengjs@ustc.edu.cn}
}

\date{}

\begin{document}
\maketitle

\begin{abstract}
Determining the dimension of the $C^1$ cubic spline space $S_3^1(\T)$ on an arbitrary nondegenerate planar triangulation has remained unresolved since the 1970s. Schumaker's lower bound includes a local correction $\sigma$ for singular interior four-stars, and it was conjectured that this bound is always attained. We disprove this conjecture by constructing a one-parameter family of nondegenerate realizations of a fixed 18-triangle complex, with only the central vertex moving as $v_6(t)=(t,0)$ on the admissible interval $I=(-3/4,24/55)$. The family exhibits three distinct cases. For $t\in I\setminus\{1/5,3/83\}$, the lower bound is attained and $\dim S_3^1(\T(t))=33$. At $t=3/83$, the central four-star is singular, $\sigma=1$, and the resulting dimension 34 is exactly accounted for by the classical local correction. At $t=1/5$, however, all interior vertices are nonsingular and $\sigma=0$, yet $\dim S_3^1(\T(1/5))=34>P_{\T(1/5)}(1,3)=33$. The smoothing-cofactor calculation shows that the dependence at $t=3/83$ is confined to the central vertex block, whereas the dependence at $t=1/5$ couples all seven interior vertex cycles even though every individual block has full row rank. A complementary Bernstein--B\'ezier calculation gives the same dimension profile. Thus the singular-four-star correction does not capture every geometry-dependent contribution to $\dim S_3^1(\T)$; genuinely global compatibility must also be taken into account.
\end{abstract}


\section{Introduction}
Let $\T$ be a finite conforming triangulation of a connected polygonal disk $\Omega\subset\R^2$. For integers $d,r\ge0$, let
\begin{equation}
S_d^r(\T)
 :=\{s\in C^r(\Omega): s|_\tau\in\Pp_d(\tau)\ \text{for every }\tau\in\T\},
\label{eq:spline-space}
\end{equation}
where $\Pp_d(\tau)$ denotes the restrictions to $\tau$ of bivariate polynomials of total degree at most $d$. In contrast to univariate spline spaces, whose dimensions are determined by the degree, smoothness, and combinatorics of the knot partition, the dimension of $S_d^r(\T)$ may also depend on the geometric realization of $\T$. A classical local manifestation of this geometry dependence is a \emph{singular} interior vertex, namely a valence-four interior vertex whose four incident edges lie on two straight lines; let $\sigma=\sigma(\T)$ denote the number of such vertices. 
This work focuses on $\dim S_3^1(\T)$ for arbitrary nondegenerate triangulations, a problem that has remained unresolved since the 1970s.

Gilbert Strang initiated the systematic study of this dimension problem by counting polynomial coefficients and smoothness constraints and asking when the resulting constraints are independent \cite{Strang1974}. Progress in the surrounding $C^1$ theory showed a sharp dependence on the polynomial degree. John Morgan and Ridgway Scott constructed a nodal basis for $S_d^1(\T)$ when $d\ge5$ \cite{MorganScott1975}; this range was later included in the arbitrary-smoothness results for $d\ge4r+1$ of Alfeld and Schumaker and for $d\ge3r+2$ of Hong \cite{AlfeldSchumaker1987,Hong1991}. Alfeld and Schumaker also settled the borderline degree $d=3r+1$ \cite{AlfeldSchumaker1990}. Thus, for $r=1$, the dimensions of $S_5^1(\T)$ and $S_4^1(\T)$ are known for arbitrary triangulations. At degree two, however, an unpublished manuscript of Morgan and Scott gave the now-classical seven-triangle example in which $\dim S_2^1(\T)$ is generically six but becomes seven when three distinguished lines are concurrent, although the combinatorial triangulation is unchanged \cite{MorganScottManuscript}. This geometry-dependent phenomenon was subsequently developed in the degree-$2r$, smoothness-$r$ regime \cite{Diener1990}.

The intervening cubic space $S_3^1(\T)$ remained exceptional. Strang's coefficient-and-constraint count supplied the expected dimension and initiated the corresponding equality conjecture \cite{Strang1974}. Schumaker subsequently incorporated the singular-vertex contribution and established a uniform lower bound valid for every geometric realization \cite{Schumaker1979,Schumaker1984}. Write $F$, $E_0$, and $V_0$ for the numbers of triangles, interior edges, and interior vertices. For every nondegenerate conforming triangulation $\T$ of a polygonal disk, the bound for $S_3^1(\T)$ is
\begin{equation}
\dim S_3^1(\T)\ge P_{\T}(1,3)
 =10+3E_0-7V_0+\sigma
 =10F-7E_0+3V_0+\sigma,
\label{eq:schumaker}
\end{equation}
where the second equality follows from $F-E_0+V_0=1$. The term $\sigma$ records a classical local geometry dependence: for a four-triangle pinwheel, moving the center from a nonsingular position onto the intersection of two opposite edge lines increases the cubic $C^1$ dimension by one \cite{Schumaker1979,AndersonMatherneTymoczko2024}.

Billera proved that equality holds in \eqref{eq:schumaker} for generic triangulations \cite{Billera1988}. Later algebraic and homological formulations identified possible discrepancies through correction modules invisible to a purely combinatorial count \cite{SchenckStillman1997,Schenck2016,SchenckStillmanYuan2020}, while structure-matrix, cell-reduction, and specialized homological arguments established equality for further restricted classes of triangulations \cite{LuoWang2006,Jaklic2022,DiPasqualeYuan2024}. Nevertheless, as late as 2024, the arbitrary-geometry cubic case was still recorded as open and no counterexample to the equality was known \cite{AndersonMatherneTymoczko2024}. The unresolved conjecture was therefore whether
\begin{equation}
\dim S_3^1(\T)=P_{\T}(1,3)
\label{eq:lower-bound-conjecture}
\end{equation}
holds for every nondegenerate planar triangulation, or equivalently whether a realization with $\sigma=0$ can carry an additional, nonlocal dimension.

This article answers the question negatively by giving, to the best of our knowledge, the first explicit nonsingular counterexample to \eqref{eq:lower-bound-conjecture}. We construct a fixed abstract triangulation with 18 triangles, 24 interior edges, and 7 interior vertices, and move only its central vertex according to $v_6(t)=(t,0)$. The complete admissible interval is $I=\left(-\frac34,\frac{24}{55}\right)$, on which every realization is a nondegenerate conforming triangulation. The exact rank profile separates the family into three mathematically distinct cases.
\begin{enumerate}[label=(\roman*)]
\item For $t\in I\setminus\{1/5,3/83\}$, all interior vertices are nonsingular, the conformality matrix has its generic rank $49$, and $\dim S_3^1(\T(t))=P_{\T(t)}(1,3)=33$. This is precisely the value predicted by Schumaker's lower bound.
\item At $t=3/83$, the central vertex is a singular four-star. A single local vertex block loses one rank, $\sigma=1$, and $\dim S_3^1(\T(3/83))=34=P_{\T(3/83)}(1,3)$. This is precisely the familiar local event already captured by Schumaker's correction.
\item At $t=1/5$, every interior vertex is nonsingular and every individual vertex block has full row rank, yet the assembled matrix again loses one rank. Its unique row dependence couples all seven interior vertex cycles, and $\dim S_3^1(\T(1/5))=34>P_{\T(1/5)}(1,3)=33$.
This is a genuinely global geometry-dependent defect and the promised counterexample to \eqref{eq:lower-bound-conjecture}.
\end{enumerate}

The remainder is organized as follows. Section~\ref{sec:matrix-models} collects the Bernstein--B\'ezier and smoothing-cofactor preliminaries used in the dimension calculation. Section~\ref{sec:family} introduces the one-parameter family and establishes its admissible geometry. Section~\ref{sec:rank-profile} proves the complete rank profile, explains the local and global rank-loss mechanisms, and records a complementary B-net check. The final section summarizes the main conclusions.

\section{Preliminaries}
\label{sec:matrix-models}

This section recalls two standard matrix methods for computing spline dimensions. The Bernstein--B\'ezier (B-net) method works directly with the $C^0$ and $C^1$ joining conditions for triangular cubic patches and provides a patch-based dimension calculation. The smoothing-cofactor method eliminates the common $C^0$ trace variables and reduces smoothness to compatibility equations around interior-vertex cycles. For the present family, the cofactor matrix is smaller, and its vertex-block structure makes the local or global origin of a rank defect explicit. We therefore use the smoothing-cofactor method to prove Theorem~\ref{thm:rank-profile}, while Remark~\ref{rem:bnet-verification} records the B-net computation as an independent verification. 

\subsection{Bernstein--B\'ezier edge matching}

Let $\tau=[a,b,c]$ be a nondegenerate triangle with barycentric coordinates $(\lambda_a,\lambda_b,\lambda_c)$. Every $p\in\Pp_3(\tau)$ has a unique Bernstein representation
\begin{equation}
p=\sum_{i+j+k=3} b^{\tau}_{ijk}
   \frac{3!}{i!j!k!}\lambda_a^i\lambda_b^j\lambda_c^k.
\label{eq:bernstein}
\end{equation}
Thus a piecewise cubic on $F$ triangles is described by $10F$ B-ordinates.

Consider two adjacent triangles
\[
\tau=[a,b,c],\qquad \tau'=[a,b,d],
\]
and write
\begin{equation}
d=\alpha a+\beta b+\gamma c,
\qquad \alpha+\beta+\gamma=1.
\label{eq:barycentric-opposite}
\end{equation}
Let $b_{ijk}$ and $\widetilde b_{ijk}$ denote the B-ordinates on $\tau$ and $\tau'$, respectively, with their first two indices associated with $a$ and $b$.

\begin{proposition}[B-net joining conditions]
\label{prop:bnet}
The two cubic pieces join with $C^1$ smoothness across $[a,b]$ if and only if
\begin{align}
\widetilde b_{ij0}&=b_{ij0}, &&i+j=3, \label{eq:c0-bnet}\\
\widetilde b_{ij1}&=\alpha b_{i+1,j,0}+\beta b_{i,j+1,0}+\gamma b_{ij1},
&&i+j=2. \label{eq:c1-bnet}
\end{align}
\end{proposition}

\begin{proof}
The four equations in \eqref{eq:c0-bnet} identify the cubic traces on the common edge. Their tangential derivatives then agree automatically. The first interior B-net row encodes a transversal derivative, and expressing the transversal direction from $c$ to $d$ through \eqref{eq:barycentric-opposite} gives the three equations in \eqref{eq:c1-bnet}. These are the standard Bernstein--B\'ezier $C^1$ joining conditions; see, for example, \cite[Chapters~3 and~5]{LaiSchumaker2007} and \cite{Alfeld2016}.
\end{proof}

Assembling the seven equations on every interior edge gives
\begin{equation}
\bbB(\T)\in\R^{7E_0\times10F}.
\label{eq:bnet-matrix}
\end{equation}
Its kernel consists exactly of the globally $C^1$ B-nets, and therefore
\begin{equation}
\dim S_3^1(\T)=10F-\rank_{\R}\bbB(\T).
\label{eq:bnet-dim}
\end{equation}

\subsection{Smoothing cofactors and vertex cycles}

A smaller matrix follows by eliminating the common traces. Give every interior edge $e$ an orientation in the dual graph, from a triangle $\tau_e^-$ to the adjacent triangle $\tau_e^+$. Let $\ell_e$ be a nonzero affine linear form whose zero set is the supporting line of $e$. Two cubics meet with $C^1$ smoothness across $e$ precisely when their difference is divisible by $\ell_e^2$; hence
\begin{equation}
p_{\tau_e^+}-p_{\tau_e^-}=\ell_e^2q_e,
\qquad q_e(x,y)=a_e+b_ex+c_ey\in\Pp_1.
\label{eq:cofactor-jump}
\end{equation}
Changing the scale of $\ell_e$ merely rescales the three cofactor variables attached to $e$ and does not alter any rank statement.

The dual graph has $F$ vertices and $E_0$ edges. Since $\T$ triangulates a disk, its cycle rank is
\[
E_0-F+1=V_0.
\]
The bounded faces of the embedded dual graph correspond to the interior vertices of $\T$, and their boundary walks form a cycle basis. Consequently, the jumps in \eqref{eq:cofactor-jump} integrate to a well-defined collection of triangle polynomials if and only if, for every interior vertex $v$,
\begin{equation}
\sum_{e\ni v}\eps_{ve}\ell_e^2q_e=0,
\label{eq:vertex-cycle}
\end{equation}
where $\eps_{ve}=\pm1$ records whether the chosen cycle around $v$ traverses the dual edge $e$ with or against its orientation.

Equation \eqref{eq:vertex-cycle} gives seven scalar conditions rather than ten. Indeed, introduce local coordinates $X=x-x_v$ and $Y=y-y_v$. If $e=[v,w]$, choose
\begin{equation}
\ell_e=A_{ve}X+B_{ve}Y,
\qquad A_{ve}=y_v-y_w,
\qquad B_{ve}=x_w-x_v.
\label{eq:edge-form}
\end{equation}
Writing $q_e=q_{e,v}+b_eX+c_eY$, with $q_{e,v}=a_e+b_ex_v+c_ey_v$, the product $\ell_e^2q_e$ belongs to
\[
\spann\{X^2,XY,Y^2,X^3,X^2Y,XY^2,Y^3\}.
\]
In this ordered local basis, the contribution of $(a_e,b_e,c_e)$ to the cycle equation at $v$ is the $7\times3$ block
\begin{equation}
H_{ve}=
\begin{pmatrix}
A^2 & A^2x_v & A^2y_v\\
2AB & 2ABx_v & 2ABy_v\\
B^2 & B^2x_v & B^2y_v\\
0 & A^2 & 0\\
0 & 2AB & A^2\\
0 & B^2 & 2AB\\
0 & 0 & B^2
\end{pmatrix},
\qquad A=A_{ve},\quad B=B_{ve}.
\label{eq:local-block}
\end{equation}
The conformality matrix is assembled from the signed blocks $\eps_{ve}H_{ve}$:
\begin{equation}
\bbC(\T)\in\R^{7V_0\times3E_0}.
\label{eq:cofactor-matrix}
\end{equation}

\begin{proposition}[Cofactor dimension formula]
\label{prop:cofactor-dim}
For any conforming triangulation of a polygonal disk,
\begin{equation}
\dim S_3^1(\T)
 =10+\dim\kernel\bbC(\T)
 =10+3E_0-\rank\bbC(\T).
\label{eq:cofactor-dim}
\end{equation}
\end{proposition}

\begin{proof}
Choose a root triangle $\tau_0$. A spline determines its root polynomial $p_{\tau_0}\in\Pp_3$ and the unique edge cofactors in \eqref{eq:cofactor-jump}. Telescoping the polynomial jumps around every dual cycle gives \eqref{eq:vertex-cycle}, so the cofactors lie in $\kernel\bbC(\T)$.

Conversely, take $p_{\tau_0}\in\Pp_3$ and a cofactor vector in $\kernel\bbC(\T)$. For any triangle $\tau$, sum the oriented jumps $\ell_e^2q_e$ along a dual path from $\tau_0$ to $\tau$. The vertex-cycle equations make this sum path independent because the interior-vertex cycles form a basis of the dual cycle space. The resulting polynomial pieces satisfy \eqref{eq:cofactor-jump} on every interior edge and therefore join with $C^1$ smoothness. This gives a linear isomorphism
\[
S_3^1(\T)\cong\Pp_3\oplus\kernel\bbC(\T),
\]
from which \eqref{eq:cofactor-dim} follows.
\end{proof}

Combining \eqref{eq:schumaker} and \eqref{eq:cofactor-dim} suggests the discrepancy
\begin{equation}
\delta(\T)
 :=\dim S_3^1(\T)-P_{\T}(1,3)
 =7V_0-\sigma-\rank\bbC(\T).
\label{eq:defect}
\end{equation}
Schumaker's bound states that $\delta(\T)\ge0$. A singular four-star contributes through $\sigma$; the example below has $\delta=1$ although $\sigma=0$.

\section{A one-parameter triangulation}
\label{sec:family}

The abstract complex has 13 vertices and 18 triangles. Seven vertices are interior:
\begin{equation}
\begin{aligned}
v_0&=(-3,0), & v_1&=(-3/4,1),
&v_2&=(24/55,28/55),\\
v_3&=(3,0), & v_4&=(24/55,-28/55),
&v_5&=(-3/4,-1),\\
v_6(t)&=(t,0),&&&
\end{aligned}
\label{eq:interior-vertices}
\end{equation}
and the boundary vertices are
\begin{equation}
\begin{aligned}
v_7&=(0,4), &v_8&=(-6,4), &v_9&=(-6,-4),\\
v_{10}&=(0,-4), &v_{11}&=(6,-4), &v_{12}&=(6,4).
\end{aligned}
\label{eq:boundary-vertices}
\end{equation}
The triangle list is
\begin{equation}
\begin{split}
&(0,1,5),(1,2,6),(2,4,6),(2,3,4),(4,5,6),(5,1,6),\\
&(7,0,1),(7,1,2),(7,2,3),(10,3,4),(10,4,5),(10,5,0),\\
&(0,7,8),(0,8,9),(0,9,10),(3,10,11),(3,11,12),(3,12,7),
\end{split}
\label{eq:triangles}
\end{equation}
where $(i,j,k)$ denotes $[v_i,v_j,v_k]$.

The vertices $v_7,\ldots,v_{12}$ lie on the boundary of the disk. They are needed to close the outer triangulation and to verify conformity, but they do not contribute vertex-cycle equations: only interior vertices correspond to closed faces of the dual graph and enter the singular-vertex count and the conformality matrix. Consequently, the later rank analysis involves only $v_0,\ldots,v_6$. Figure~\ref{fig:mesh-family} shows the complete triangulation at a representative regular parameter and compares the two central configurations at which the rank drops.

\begin{figure}[htbp]
\centering
\includegraphics[width=\linewidth]{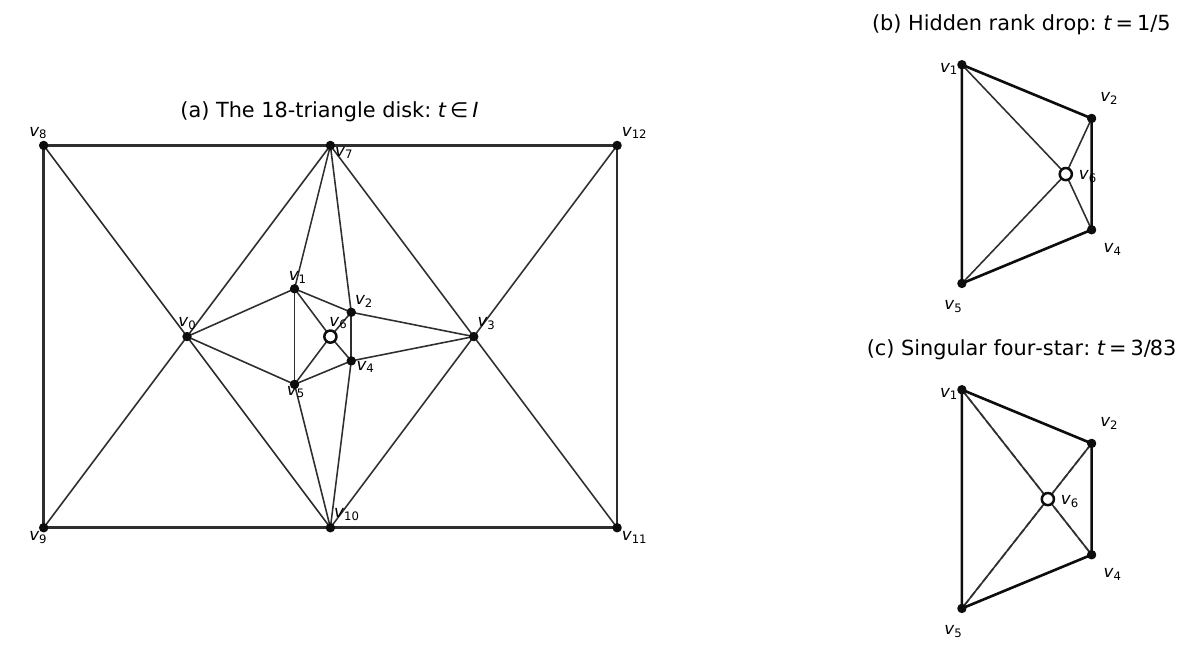}
\caption{The one-parameter family. (a) The complete 18-triangle disk at the representative value $t=0\in I$. The moving vertex is shown as an open circle. (b) At $t=1/5$, the four incident directions are distinct although the global conformality matrix loses one rank. (c) At $t=3/83$, the opposite edge pairs become collinear and the center is a singular four-star.}
\label{fig:mesh-family}
\end{figure}

\begin{lemma}[Admissible interval]
\label{lem:admissible}
For every
\begin{equation}
-\frac34<t<\frac{24}{55},
\qquad I:=\left(-\frac34,\frac{24}{55}\right),
\label{eq:interval}
\end{equation}
the coordinates \eqref{eq:interior-vertices}--\eqref{eq:boundary-vertices} and the triangle list \eqref{eq:triangles} define a nondegenerate conforming triangulation of the same polygonal disk. Throughout this interval,
\begin{equation}
F=18,\qquad E_0=24,\qquad V_0=7.
\label{eq:counts}
\end{equation}
\end{lemma}

\begin{proof}
The vertices $v_1,v_2,v_4,v_5$ form a convex quadrilateral $Q$, and its intersection with the $x$-axis is the segment joining $(-3/4,0)$ to $(24/55,0)$. Hence $v_6(t)$ lies in the interior of $Q$ precisely for $t\in I$. In that case the four moving triangles form the nonoverlapping fan of $Q$ about $v_6(t)$.

For completeness, their oriented doubled areas are
\begin{align}
[1,2,6]&=\frac{9(3t-5)}{55},
&[2,4,6]&=\frac{56(55t-24)}{3025},\nonumber\\
[4,5,6]&=\frac{9(3t-5)}{55},
&[5,1,6]&=-\frac{4t+3}{2},
\label{eq:moving-areas}
\end{align}
where $[i,j,k]=\det(v_j-v_i,v_k-v_i)$. All four expressions are nonzero and have the same sign on $I$. The remaining triangles and edges are fixed. Direct exact calculation from \eqref{eq:interior-vertices}--\eqref{eq:triangles} shows that every fixed triangle is nondegenerate, nonadjacent fixed edges have disjoint relative interiors, and the fixed triangle areas sum to the area between $Q$ and the boundary hexagon
\[
v_7-v_8-v_9-v_{10}-v_{11}-v_{12}-v_7.
\]
Thus replacing the interior fan point by any other point of the indicated segment cannot create a crossing, overlap, or gap. The edge set consists of 30 edges, six on the boundary and 24 in the interior, which gives \eqref{eq:counts}.
\end{proof}

\begin{lemma}[Singular-vertex count]
\label{lem:sigma}
For the family in Lemma~\ref{lem:admissible},
\begin{equation}
\sigma(t)=
\begin{cases}
1,&t=3/83,\\
0,&t\in I\setminus\{3/83\}.
\end{cases}
\label{eq:sigma-profile}
\end{equation}
\end{lemma}

\begin{proof}
The valences of $v_0,\ldots,v_6$ are $(6,5,5,6,5,5,4)$, so only $v_6$ can be a singular four-star. The determinants for the two required pairs of opposite directions coincide:
\begin{equation}
\det(v_1-v_6,v_4-v_6)
 =\det(v_2-v_6,v_5-v_6)
 =\frac{83t-3}{55}.
\label{eq:singular-determinant}
\end{equation}
They hold simultaneously exactly at $t=3/83$. At every other $t\in I$, the four incident edges have four distinct projective directions.
\end{proof}

Using \eqref{eq:counts}, the lower-bound expression is
\begin{equation}
P_{\T(t)}(1,3)=33+\sigma(t).
\label{eq:count-family}
\end{equation}
The issue is therefore whether the cofactor matrix has full row rank 49 whenever $\sigma=0$.

\section{The exact rank profile}
\label{sec:rank-profile}

Assemble the conformality matrix $\bbC(t)\in\R[t]^{49\times72}$ from \eqref{eq:local-block}. Rows are grouped into seven-row blocks for $v_0,\ldots,v_6$, ordered as
\begin{equation}
X^2,\ XY,\ Y^2,\ X^3,\ X^2Y,\ XY^2,\ Y^3,
\label{eq:row-order}
\end{equation}
and columns are grouped into triples $(a_e,b_e,c_e)$ for the 24 interior edges in lexicographic edge order. The triangles are numbered from 0 to 17 in the order displayed in \eqref{eq:triangles}; each dual edge is oriented from the lower to the higher triangle index, and every vertex cycle is traversed counterclockwise. These conventions are fixed throughout the certificates below.

\begin{theorem}[Complete rank profile]
\label{thm:rank-profile}
For every $t\in I$,
\begin{equation}
\rank\bbC(t)=
\begin{cases}
48,&t=1/5\ \text{or}\ t=3/83,\\
49,&t\in I\setminus\{1/5,3/83\}.
\end{cases}
\label{eq:rank-profile}
\end{equation}
Consequently,
\begin{equation}
\dim S_3^1(\T(t))=
\begin{cases}
34,&t=1/5\ \text{or}\ t=3/83,\\
33,&t\in I\setminus\{1/5,3/83\}.
\end{cases}
\label{eq:dimension-profile}
\end{equation}
\end{theorem}

\begin{proof}
The proof uses two stages of exact rank certificates. First, we exhibit a nonzero maximal minor of $\bbC(t)$; its factorization proves full row rank except at the admissible zeros of that minor. Second, at each remaining parameter, a nonzero vector $z$ satisfying $z^{\mathsf T}\bbC(t)=0$ proves that the rank drops, while a nonzero $48\times48$ minor proves that it drops by exactly one. The spline dimensions then follow from Proposition~\ref{prop:cofactor-dim}.

Select the 49 columns
\begin{equation}
J=\{1{:}12,\ 19{:}51,\ 58,59,61,62\},
\label{eq:J}
\end{equation}
using one-based indexing and the convention above. Exact elimination gives
\begin{equation}
\det\bbC(t)_{[:,J]}
 =\kappa(3t-5)^4(4t+3)(5t-1)(55t-24)^3(83t-3),
\label{eq:maximal-minor}
\end{equation}
where
\begin{equation}
\kappa=\frac{2^{104}3^{75}7^{11}29^3 47^5}{5^{34}11^{38}}\ne0.
\label{eq:kappa}
\end{equation}
On the open interval $I$, the factors $3t-5$, $4t+3$, and $55t-24$ do not vanish. The maximal minor is therefore nonzero except at $t=1/5$ and $t=3/83$. Since $\bbC(t)$ has 49 rows, this proves full row rank at every other admissible parameter.

At each of the two remaining parameters, exact elimination produces an explicit nonzero vector $z$ satisfying $z^{\mathsf T}\bbC(t)=0$, so the rank is at most 48. At $t=1/5$, the $48\times48$ minor with rows $\{1{:}45,47{:}49\}$ and columns $\{1{:}12,19{:}51,58,59,61\}$ is nonzero. At $t=3/83$, the minor with rows $\{1{:}44,46{:}49\}$ and the same columns is nonzero. Thus the rank is at least 48 at both points and \eqref{eq:rank-profile} follows. Formula \eqref{eq:dimension-profile} is then immediate from Proposition~\ref{prop:cofactor-dim}, since $10+3E_0=82$.
\end{proof}

The factorization in \eqref{eq:maximal-minor} separates three types of roots. The endpoint factors $4t+3$ and $55t-24$ correspond to a collapsing central fan; $3t-5$ lies outside the admissible interval. The two interior roots are the only nondegenerate rank-loss events in this family.

\begin{corollary}[A nonsingular counterexample]
\label{cor:counterexample}
At $t=1/5$, the triangulation is nondegenerate, $\sigma=0$, and
\begin{equation}
\dim S_3^1(\T(1/5))=34>P_{\T(1/5)}(1,3)=33.
\label{eq:strict-counterexample}
\end{equation}
\end{corollary}

To display the coordinate dependence on a fixed abstract triangulation, Table~\ref{tab:rank-dim} compares the hidden point $t=1/5$ with the nearby realization $t=101/500=1/5+1/500$. Both realizations are nondegenerate and have $\sigma=0$, but this small perturbation raises the matrix rank from 48 to 49 and changes the spline dimension from 34 to 33 without changing the combinatorial complex.

\begin{table}[htbp]
\centering
\caption{Local count, matrix rank, and dimension at representative parameter values.}
\label{tab:rank-dim}
\begin{tabular}{@{}llllll@{}}
\toprule
Parameter & Geometry at $v_6$ & $\sigma$ & $\rank\bbC$ & $P_{\T}(1,3)$ & $\dim S_3^1$ \\
\midrule
$t\in I\setminus\{1/5,3/83\}$ & nonsingular & 0 & 49 & 33 & 33 \\
$t=101/500$ & nonsingular, $1/500$ from $1/5$ & 0 & 49 & 33 & 33 \\
$t=3/83$ & singular four-star & 1 & 48 & 34 & 34 \\
$\boldsymbol{t=1/5}$ & \textbf{nonsingular} & \textbf{0} & \textbf{48} & \textbf{33} & \textbf{34} \\
\bottomrule
\end{tabular}
\end{table}

After removing the factors that stay nonzero in the interior of $I$, the two roots in \eqref{eq:maximal-minor} are simple. Thus each exceptional point is isolated along this parameter line, and the selected maximal minor vanishes there to first order. In particular, the additional spline mode disappears under an arbitrarily small generic displacement of the moving vertex; it is not a combinatorial degree of freedom attached to the abstract complex.

\begin{remark}[Local and global rank loss]
\label{rem:local-global}
The equal dimensions at $t=1/5$ and $t=3/83$ arise from different mechanisms. Let $\bbC_v(t)$ denote the seven-row block associated with the cycle around $v$. At $t=3/83$, the central block has rank 6 while the other six blocks have rank 7, and the left-kernel relation is supported only on the central four-star. This is the familiar local dependence counted by $\sigma=1$. At $t=1/5$, every block $\bbC_v(1/5)$ has rank 7, but the assembled matrix has rank 48 because its unique row dependence has nonzero components in all seven vertex-cycle blocks. The latter dependence is therefore global: it appears only when the individually full-rank local systems are assembled. The same one-dimensional matrix rank loss is absorbed by the local correction at $t=3/83$, whereas at $t=1/5$ it remains as the additional dimension beyond Schumaker's lower bound.
\end{remark}

\begin{remark}[Complementary B-net calculation]
\label{rem:bnet-verification}
The dimension profile can also be recovered directly from the larger Bernstein--B\'ezier matching matrix. For this family, $\bbB(t)\in\R(t)^{168\times180}$. Direct rank computation gives rank 146 at $t=1/5$ and $t=3/83$, and rank 147 for $t\in I\setminus\{1/5,3/83\}$. Formula~\eqref{eq:bnet-dim} therefore yields dimensions 34 at the two exceptional parameters and 33 elsewhere, in agreement with the smoothing-cofactor calculation.
\end{remark}

\section{Conclusion}

We have disproved the long-standing conjecture that Schumaker's lower bound is always attained for $S_3^1(\T)$ on nondegenerate planar triangulation by constructing a one-parameter family of realizations of a fixed 18-triangle complex,
with  the central vertex moving as $v_6(t)=(t,0)$ on the admissible interval $I=(-3/4,24/55)$.
For $t\in I\setminus\{1/5,3/83\}$, the lower bound is attained and $\dim S_3^1(\T(t))=33$. At $t=3/83$, the central four-star is singular and the local correction $\sigma=1$ accounts for the dimension 34. At $t=1/5$, by contrast, all interior vertices are nonsingular, but $\dim S_3^1(\T(1/5))=34>P_{\T(1/5)}(1,3)=33$. The smoothing-cofactor analysis distinguishes the two exceptional mechanisms: the dependence at $t=3/83$ is confined to the central vertex block, whereas the dependence at $t=1/5$ couples all seven interior vertex cycles even though every individual block has full row rank. The complementary Bernstein--B\'ezier calculation gives the same dimension profile. Thus, for arbitrary nondegenerate realizations, the singular-four-star correction does not capture every geometry-dependent contribution to the dimension of $S_3^1$; genuinely global compatibility must also be taken into account.


\section*{Acknowledgements}

This work was supported by the National Key R\&D Program of China (Nos. 2022YFA1005200 and 2022YFA1005201), the NSF of China (No. 12171453), and the Major Project of Science and Technology Innovation Tackling Plan of Anhui Province (No. 202423e09050003).

\section*{Declarations}

The authors declare that they have no conflict of interest.

\bibliographystyle{cas-model2-names}

\bibliography{refs}



\end{document}